\documentclass[10pt,english,online]{smfart}
\usepackage[T1]{fontenc}
\usepackage[francais]{babel}
\usepackage{pb-diagram}
\usepackage[all]{xy}
\usepackage{amssymb,url,xspace,smfthm,amsmath,amsfonts,enumitem}

\usepackage{tikz}
\usetikzlibrary{matrix,positioning}
\title[Logarithmic good reduction of abelian varieties]{Logarithmic good reduction \\ of abelian varieties}
\author{Alberto Bellardini} 
\address{University of Leuven, Department Wiskunde, Celestijnenlaan 200B, 3001 Heverlee, Belgium}
\email{albertobellardini@yahoo.it}
\author{Arne Smeets}
\address{Radboud Universiteit Nijmegen, IMAPP, Heyendaalseweg 135, 6525 AJ Nijmegen, The Netherlands \emph{and} University of Leuven, Departement Wiskunde, Celestijnenlaan 200B, 3001 Heverlee, Belgium} 
\email{arnesmeets@gmail.com}

\begin{document}
\maketitle

\begin{abstract}
Let $K$ be a field which is complete for a discrete valuation. We prove a logarithmic version of the N\'eron--Ogg--Shafarevich criterion: if $A$ is an abelian variety over $K$ which is cohomologically tame, then $A$ has good reduction in the logarithmic setting, i.e. there exists a projective, log smooth model of $A$ over $\mathcal{O}_K$. This implies in particular the existence of a projective, regular model of $A$, generalizing a result of K\"unnemann. The proof combines a deep theorem of Gabber with the theory of degenerations of abelian varieties developed by Mumford, Faltings--Chai \emph{et al}.
\end{abstract}

\section{Introduction}

\subsection{Notation} Let $K$ be a field complete for a discrete valuation, with residue field $k$ of characteristic $p$. Let $\mathcal{O}_K$ be its ring of integers, with spectrum $S$, generic point $\eta$ and closed point $s$. Fix a separable closure $K^s$ of $K$ and let $K^t$ be the maximal extension of $K$ inside $K^s$ which is tamely ramified. Let $P = \mathrm{Gal}(K^s/K^t)$ be the wild inertia subgroup of the Galois group $\mathrm{Gal}(K^s/K)$. Fix a prime $\ell$ different from the residual characteristic $p$.

\subsection{Question} Let $\mathcal{X}$ be a proper, flat $S$-scheme, and $X = \mathcal{X}_{\eta}$. We equip $\mathcal{X}$ with the natural log structure induced by its special fibre $\mathcal{X}_s$. Let $\mathcal{X}^\dag$ denote the resulting log scheme: we refer to \cite{katologstructures} for the basic terminology. In particular, $S^\dagger$ denotes $S$ equipped with the log structure induced by $s$. A beautiful result of Nakayama \cite[Corollary 0.1.1]{Nak} says that if the morphism of log schemes $\mathcal{X}^\dag \to S^\dag$ is \emph{log smooth}, then the action of the wild inertia group $P$ on $H^\star(X \times_K K^s, \mathbf{Q}_\ell)$ is trivial; a variety with the latter property will be said to be \emph{cohomologically tame}. It is natural to ask when the converse holds: given a smooth, proper variety $X$ over $K$ which is cohomologically tame, can one conclude that there exists a proper model $\mathcal{X}$ over $S$ which is log smooth?  Of course this question can only be interesting in the case of mixed or positive characteristic, since both cohomological tameness and the existence of a proper, log smooth model are automatic if the residual characteristic is zero.

\subsection{Results}  The answer to the question raised above is known to be affirmative when $X$ is a curve of genus at least $2$, due to the work of T. Saito \cite{Sai, saitologsmred} and Stix \cite{Sti}. On the other hand, the answer to the question is already negative for some genus $1$ curves:

\begin{exem} \label{example} Assume that $k$ is algebraically closed of characteristic $p > 0$. Let $E$ be an elliptic curve over $K$ with semistable bad reduction. Let $X$ be a torsor under $E$ of period $p$; by \cite[Theorem 6.6, Corollary 6.7]{LLR}, such a torsor exists, and moreover the ``type'' of the special fibre $\mathcal{X}_s$ of its minimal regular proper model $\mathcal{X}$ with strict normal crossings is obtained by multiplying the corresponding type for $E$ by $p$. Then $X$ is certainly cohomologically tame (since $E$ is), but the argument in \cite[Proposition 5.2]{Sti} shows that $X$ does \emph{not} have logarithmic good reduction in this case. Indeed, the morphism $\mathcal{X}^\dagger \to S^\dagger$ is not log smooth at the points in $\mathcal{X}_s$ where the log structure has rank $2$. \end{exem}

It is not clear what to expect in higher dimension, but note that proving the existence of a proper, log smooth model is really proving a version of resolution of singularities: the existence of such a model implies the existence of a proper, regular model with strict normal crossings by \cite[\S 10.4]{kato}. Hence this problem is likely to be hard in higher dimension! 

The goal of this short note is to prove such a result for abelian varieties. Our result can be thought of as a logarithmic version of the N\'eron--Ogg--Shafarevich criterion proven by Serre--Tate \cite[Theorem 1]{st}, and it is the natural complement for abelian varieties to the work of T. Saito and Stix on curves mentioned above. Of course only one implication is new (as mentioned above, the other one is \cite[Corollary 0.1.1]{Nak}):

    \begin{theo} \label{maintheorem} Let $A$ be an abelian variety over $K$. Then $A$ is cohomologically tame if and only if there exists a projective model $\mathcal{X}$ of $A$ over $S$ such that $\mathcal{X}^\dagger \to S^\dagger$ is log smooth. \end{theo}

For abelian varieties, the cohomological tameness condition is equivalent to the Galois action on the $\ell$-adic Tate module being tamely ramified. It is classical that this is the case whenever $p$ is sufficiently large. In fact, $p > 2 \dim A + 1$ suffices; this can be deduced from Grothendieck's results \cite[Proposition 3.5, Th\'eor\`eme 4.3]{grothendieck} (see \cite[\S 3]{Loerke} for details). Our result does not extend to arbitrary torsors under abelian varieties, as Example \ref{example} illustrates.
 
    As an immediate corollary of our result, we obtain the following generalization of an important result of K\"unnemann \cite[Theorem 3.5]{ku}, who proves the existence of projective, regular models in the case of semi-abelian reduction, i.e. unipotent Galois action: 
    
        \begin{coro} Let $A$ be a cohomologically tame abelian variety over $K$. Then there exists a projective regular model $\mathcal{X}$ of $A$ over $S$, with $\mathcal{X}_s$ a strict normal crossings divisor. \end{coro} 

\begin{proof} Since the model $\mathcal{X}^\dag$ from Theorem \ref{maintheorem} is log smooth over the log regular scheme $S^\dag$, it is log regular \cite[Theorem 8.2]{kato}. Hence $\mathcal{X}^\dag$ can be desingularized using 
log blow-ups (see \cite[\S 10.4]{kato} and \cite[\S 5.3]{nizi}), yielding the desired regular model. \end{proof}

    \begin{rema} The original motivation for this note lies in the recent paper \cite{smeets}, where the second author proved a trace formula conjectured by Nicaise for the class of varieties having ``logarithmic good reduction''. Such a formula had been proven for cohomologically tame semi-abelian varieties by Halle--Nicaise \cite[\S 8.1]{halnic}. Hence Theorem \ref{maintheorem} shows that their result in the case of abelian varieties is a special case of the main result of \cite{smeets}. \end{rema}

\subsection{Conventions} We work with log structures in the \'etale topology. The theory of log regularity and desingularization by log blow-ups has been developed by Kato for Zariski log structures \cite{kato}; however, Nizio\l\ \cite{nizi} showed that starting from an \'etale log structure, one can essentially produce a Zariski log structure using suitable log blow-ups. For our purposes, we can (and will) use Kato's theory in the \'etale setting, without further mention.

\subsection{Acknowledgements} We are grateful to Dan Abramovich, Lars Halle, Luc Illusie, Klaus K\"unnemann, Johannes Nicaise, Martin Olsson, Takeshi Saito and Heer Zhao for useful discussions, and to the referee for a careful reading. We acknowledge the support of the European Research Council's FP7 programme under ERC Grant Agreements $\sharp$306610 (MOTZETA, J. Nicaise) and $\sharp$615722 (MOTMELSUM, R. Cluckers). The second-named author is a postdoctoral fellow of FWO Vlaanderen (Research Foundation -- Flanders).

\normalsize
\section{Proof of the main result}

Let us now prove the main result. The essential observation is that a deep theorem of Gabber (written down by Illusie--Temkin in \cite{illusie2}) on quotients of varieties equipped with a tame group action combines very well with the theory of degenerations of abelian varieties due to Mumford and Faltings--Chai, especially in its ``equivariant'' version developed by K\"unnemann \cite{ku}. We keep the notation used in the introduction; in particular, we assume that $A$ is a cohomologically tame abelian variety. We start with a very  easy observation:

    \begin{lemm}\label{tamextension} There exists a tame extension $L/K$ such that $A_L$ has semistable reduction, i.e. there exists a semi-abelian scheme $\mathcal{A}$ which is a model of $A_L$ over $\mathcal{O}_L$.
    \end{lemm}
    \begin{proof} Since the $\mathrm{Gal}(K^s/K)$-action on $H^1(A \times_K K^s, \mathbf{Q}_\ell)$ is tamely ramified, there exists a tamely ramified extension $L/K$ such that the action of $\mathrm{Gal}(K^s/L)$ on $H^1(A \times_K K^s,\mathbf{Q}_\ell)$ is unipotent. A famous theorem of Grothendieck \cite[Proposition 3.5, Corollaire 3.8]{grothendieck} implies that the abelian variety $A_L$ has a semi-abelian model $\mathcal{A}$ over $\mathcal{O}_L$. \end{proof}
    
Let us denote $\Gamma = \mathrm{Gal}(L/K)$ and $S' = \mathrm{Spec}\,\mathcal{O}_L$, with $L$ as in the lemma. The following result has been proven essentially (in a different language) by K\"unnemann \cite[Theorem 3.5]{ku}, building on the work of Mumford and Faltings--Chai.

\begin{theo}\label{def:compactification} There exists a flat, projective $S'$-scheme $\mathcal{P}$ such that
\begin{enumerate}
\item[(1)] we have an open immersion $\mathcal{A} \hookrightarrow \mathcal{P}$, and the natural $\Gamma$-action on $\mathcal{A}$ extends to $\mathcal{P}$, inducing a $\Gamma$-equivariant isomorphism on the generic fibres;
\item[(2)] the log scheme $\mathcal{P}^\dag$ is an fs log scheme, log smooth over $(S')^\dag$;
\item[(3)] finally, $\mathcal{P}$ admits an ample, $\Gamma$-equivariant line bundle $\mathcal{L}_{\mathcal{P}}$.
\end{enumerate}
\end{theo}

We now explain why this follows from \cite[Theorem 3.5]{ku}. We refer to the beautifully written paper \cite{ku} for more background, details and terminology.

We start with a line bundle $\mathcal{L}$ on $\mathcal{A}$ which is symmetric, $S'$-ample and $\Gamma$-equivariant. To construct it, choose any ample, symmetric line bundle on $A$, and consider its pullback to $A_L$; a sufficiently high power of the latter will extend to a line bundle on $\mathcal{A}$ with all of the desired properties. The pair $(\mathcal{A},\mathcal{L})$ determines the so-called Raynaud extension \cite[II, \S 1]{fc}: a global extension (over $S'$) given by a short exact sequence \begin{equation} 0\longrightarrow \mathcal{T} \longrightarrow \widetilde{\mathcal{G}}\stackrel{\pi}{\longrightarrow} \mathcal{B} \longrightarrow 0, \label{raynaud} \end{equation} where $\mathcal{T}$ is a torus of constant rank and $\mathcal{B}$ is an abelian scheme. We may and will assume $\mathcal{T}$ to be split: this can be achieved using a finite \'etale base change, which is completely harmless for our purposes (as will be clear from what follows). The line bundle $\mathcal{L}$ induces a cubical, ample line bundle $\widetilde{\mathcal{L}}$ on $\widetilde{\mathcal{G}}$ which, under our assumptions, descends to a non-unique ample line bundle $\mathcal{M}$ on $\mathcal{B}$ (i.e. $\pi^\star \mathcal{M} = \widetilde{\mathcal{L}}$). Fixing such a line bundle yields an object $(\mathcal{A},\mathcal{L},\mathcal{M})$ in the category $\mathrm{DEG}_{\mathrm{ample}}^{\mathrm{split}}$, in the notation of \cite[\S 2.1]{ku}. One easily checks that the Galois group $\Gamma$ acts on this object by morphisms in $\mathrm{DEG}_{\mathrm{ample}}^{\mathrm{split}}$.

In \cite[II, \S 6.2]{fc}, Faltings--Chai construct an equivalence of categories $\mathrm{F}$ between $\mathrm{DEG}_{\mathrm{ample}}^{\mathrm{split}}$ and the category $\mathrm{DD}_{\mathrm{ample}}^{\mathrm{split}}$ of \emph{split, ample degeneration data}; a quasi-inverse for $\mathrm{F}$ is given by \emph{Mumford's construction} (see \cite[III]{fc} and \cite[\S 2.13]{ku}). Hence we can associate to the triple $(\mathcal{A},\mathcal{L},\mathcal{M})$ a tuple \begin{equation} \label{degdata} \mathrm{F}(\mathcal{A},\mathcal{L},\mathcal{M}) := (\mathcal{B},X,Y,\phi,c,c^t,\widetilde{\mathcal{G}},\iota,\tau,\widetilde{\mathcal{L}},\mathcal{M},\lambda_{\mathcal{B}},\psi, a,b)\end{equation} of degeneration data. We refer to \cite[\S 2.2]{ku} for details, but recall here that $X$ and $Y$ are lattices of rank equal to the toric rank of $\mathcal{T}$, equipped with a $\Gamma$-action satisfying various natural compatibility conditions, cfr. \cite[\S 2.4]{ku}. 

Let $M = X \oplus \mathbf{Z}$, and let $N = X^\vee \oplus \mathbf{Z}$ be the dual lattice. The group $Y \rtimes \Gamma$ acts on the vector space $N_{\mathbf{R}}$ (see \cite[\S 3.1]{ku}). We need a decomposition of the cone $$\mathcal{C} = X^\vee_\mathbf{R} \times \mathbf{R}_{> 0} \cup \{0\}$$ into rational polyhedral cones which contains $\{0\} \times \mathbf{R}_{> 0}$, is $(Y \rtimes \Gamma)$-admissible and admits a $k$-twisted $(Y \rtimes \Gamma)$-admissible polarization function, for some $k \in \mathbf{Z}_{>0}$. Let $\Sigma$ be such a decomposition, the existence of which follows from \cite[Proposition 3.3]{ku}.

\begin{rema} K\"unnemann actually needs a decomposition into \emph{smooth} cones, for the construction of regular models. However, any $\Sigma$ as above will suit our needs. \end{rema}

We will now verify that K\"unnemann's method (applied to $\Sigma$) yields a model with the properties listed in Theorem \ref{def:compactification}. The only thing we need to check is the log smoothness property; everything else works as in the proof of \cite[Theorem 3.5]{ku}.

Let $\sigma \in \Sigma$. The projection $N = X^\vee \oplus \mathbf{Z} \to \mathbf{Z}$ onto the second factor then yields a homomorphism of monoids $f_\sigma: \sigma \cap N \to \mathbf{N}$, with dual $f_\sigma^\vee: \mathbf{N} \to \sigma^\vee \cap M$. Fix a uniformizer $\pi$ of $\mathcal{O}_L$ and consider the homomorphism $\mathbf{Z}[\mathbf{N}] \to \mathcal{O}_L: 1 \mapsto \pi$. Define \begin{eqnarray*} Z(\sigma) & = & \mathrm{Spec}\left(\mathcal{O}_L \otimes_{\mathbf{Z}[\mathbf{N}]} \mathbf{Z}[\sigma^\vee \cap M]\right) \\ & = & \mathrm{Spec}\left(\mathcal{O}_L[\sigma^\vee \cap M]/(f_\sigma^\vee(1) - \pi)\right).\end{eqnarray*} The homomorphism of monoids $\sigma^\vee \cap M \to \mathcal{O}_L[\sigma^\vee \cap M]/(f_\sigma^\vee(1) - \pi)$ defines a natural fs log structure on the toric scheme $Z(\sigma)$, yielding the fs log scheme $Z(\sigma)^\dag$. 
\begin{lemm} \label{logsmooth} The resulting map of fs log schemes $Z(\sigma)^\dag \to (S')^\dag$ is log smooth. \end{lemm} 
\begin{proof} This follows from Kato's criterion \cite[Theorem 3.5]{katologstructures}. Indeed, the map $$(f_\sigma^\vee)^\mathrm{gp}: \mathbf{Z} \to (\sigma^\vee \cap M)^\mathrm{gp}$$ is injective. The group $(\sigma^\vee \cap M)^\mathrm{gp}$ is a subgroup of $X \oplus \mathbf{Z}$ which contains $\{0\} \oplus \mathbf{Z}$. An easy calculation shows that $(f_\sigma^\vee)^\mathrm{gp}$ simply maps $\mathbf{Z}$ isomorphically onto the last factor. Therefore the cokernel of $(f_\sigma^\vee)^\mathrm{gp}$ is torsion free. \end{proof}

The toric schemes $Z(\sigma)$ (as $\sigma$ varies) glue according to the inclusion relations between different elements of $\Sigma$, yielding an $S'$-scheme $Z$ such that $Z^\dag$ is log smooth over $(S')^\dag$ (by the lemma and the Zariski local nature of log smoothness). Since $\Sigma$ admits a suitable polarization function, we get a $\mathcal{T}$-linearized ample line bundle on $Z$ which we denote by $\mathcal{N}$. The pair $(Z,\mathcal{N})$ inherits an action of $\Gamma$. We take $$\widetilde{\mathcal{P}} = \widetilde{\mathcal{G}} \times^{\mathcal{T}} Z \stackrel{\widetilde{\pi}}{\longrightarrow} \mathcal{B},\ \ \widetilde{\mathcal{L}}_{\widetilde{\mathcal{P}}} = \widetilde{\mathcal{L}} \times^\mathcal{T} \mathcal{N}$$ (with the notation of (\ref{raynaud}), (\ref{degdata}) and \cite[\S 1.19]{ku}). 

K\"unnemann proves that  $(\widetilde{\mathcal{P}},\widetilde{\mathcal{L}}_{\widetilde{\mathcal{P}}})$ is a relatively complete model as in \cite[III, \S 3]{fc} for the data (\ref{degdata}), and that the $\Gamma$-action extends to this model. We have to check that $\widetilde{\mathcal{P}}^\dag$ is log smooth over $(S')^\dag$: this follows from Lemma \ref{logsmooth} and the fact that Zariski locally on $\mathcal{B}$, the contracted product $\widetilde{\mathcal{G}} \times^{\mathcal{T}} Z$ looks like $\mathcal{B} \times_{S'} Z$.

To construct the pair $(\mathcal{P},\mathcal{L}_{\mathcal{P}})$ as in Theorem \ref{def:compactification}, we apply Mumford's construction to $(\widetilde{\mathcal{P}},\widetilde{\mathcal{L}}_{\widetilde{\mathcal{P}}})$, as in \cite[\S 3.8]{ku}. Again the only additional statement to check is that $\mathcal{P}^\dag$ is log smooth over $(S')^\dag$. It suffices to check that the fibres are log smooth, by the lemma in the appendix. For the generic fibre this is clear. Concerning the special fibre, we have an \'etale, strict morphism of formal log schemes $\widehat{\widetilde{\mathcal{P}}^\dag} \to \widehat{\mathcal{P}^\dag}.$ Restricting to the closed point $s'$ yields an \'etale, strict morphism $\widetilde{\mathcal{P}}^\dag_{s'} \to \mathcal{P}^\dag_{s'}$. Since $\widetilde{\mathcal{P}}^\dag_{s'}$ is log smooth over $(s')^\dag = (s',\mathbf{N})$, so is $\mathcal{P}^\dag_{s'}$.

These arguments (together with all of the work in \cite[\S 3]{ku}) prove Theorem \ref{def:compactification}. \hfill \qed

\vskip 7pt

We will now use Gabber's theorem to prove our main result. Let us briefly recall the notion of a \emph{very tame group action} on a log scheme, which is the essential player in Gabber's result. As in \cite[\S 3.1]{illusie1}, an action of a finite group $\Gamma$ on an fs log scheme $(X,\mathcal{M}_X)$ is said to be \emph{tame} if for every geometric point $\overline{x}$ lying over a point $x$ of $X$, the order of the stabilizer $\Gamma_{\overline{x}}$ is prime to the characteristic of the residue field at $x$. The action is said to be \emph{very tame} if it is tame, and for every geometric point $\overline{x}$, the following conditions are satisfied: $\Gamma$ acts trivially on $\overline{\mathcal{M}}_{X,\overline{x}} = \mathcal{M}_{X,\overline{x}}/\mathcal{O}_{X,\overline{x}}^\times$, and  $\Gamma_{\overline{x}}$ acts trivially on the scheme $C_{X,\overline{x}} = \mathrm{Spec}\,\mathcal{O}_{X,\overline{x}}/I(\overline{x},\mathcal{M}_X)$, where as usual $I(\overline{x},\mathcal{M}_X)$ denotes the ideal of $\mathcal{O}_{X,\overline{x}}$ generated by $\mathcal{M}_{X,\overline{x}} \setminus \mathcal{O}_{X,\overline{x}}^\times$.



\begin{proof}[End of proof of Theorem 1.1]
Let $\mathcal{P}$ be as in Theorem \ref{def:compactification}; we can forget $\mathcal{L}_{\mathcal{P}}$ at this stage. Since the morphism $(S')^\dag \to S^\dag$ is log \'etale, the composition $\mathcal{P}^\dag \to (S')^\dag \to S^\dag$ is a log smooth morphism of fs log schemes endowed with a $\Gamma$-action. 

The hypotheses of \cite[Theorem 1.1]{illusie2} are satisfied. Note that $\Gamma = \mathrm{Gal}(L/K)$ can have order divisible by $p$: one may need a tamely ramified extension, with (necessarily separable) residue field extension of degree divisible by $p$, for $\mathcal{T}$ in the proof of Theorem 2.2 to become split. But $\Gamma$ acts tamely since the elements of $\Gamma$ which stabilize a \emph{geometric} point of the special fibre lie in the inertia subgroup, which does have order prime to $p$.

Therefore \cite[loc. cit.]{illusie2} yields the existence of a projective modification $h: \mathcal{P}_1 \to \mathcal{P}$ such that $\mathcal{P}_1^\dag$ is an fs log scheme, log smooth over $(S')^\dag$ and the $\Gamma$-action on $\mathcal{P}_1^\dag$ is now \emph{very} tame. This modification does not affect the generic fibre of $\mathcal{P}$. The quotient $\mathcal{X} = \mathcal{P}_1/\Gamma$ exists as a scheme since $\mathcal{P}_1$ is projective over $S'$ (e.g. \cite[Lemma 3.2]{saitologsmred}), and the resulting fs log scheme $\mathcal{X}^\dag$ is log smooth over $S^\dag$ by Gabber's theorem. Hence $\mathcal{X}$ is our desired model. \end{proof}

\section*{Appendix: a fibrewise criterion for log smoothness}

We used the following lemma, which is more general than needed. It may however be of independent interest (we could not locate such a result in the literature). 

\vskip 10pt
\noindent \textbf{\textit{Lemma. --- }} \textit{Let $f :(X,\mathcal{M}_X)\rightarrow (T,\mathcal{M}_T)$ be an integral morphism of fine log schemes with log smooth fibres, such that the morphism of underlying schemes $X \to T$ is locally of finite presentation and flat. Then $f$  is log smooth.}
\vskip 10pt
To prove this, we use Olsson's theory of stacks of log structures \cite{olsson}.

\begin{proof}
By \cite[4.6.(ii)]{olsson}, it suffices to show that the induced morphism 
\begin{equation}\label{lemm:logmap}
\mathcal{L}og(f):\mathcal{L}og_X\rightarrow \mathcal{L}og_T
\end{equation}
of algebraic stacks is formally smooth. By our assumptions, it is locally of finite presentation; hence it suffices to show that (\ref{lemm:logmap}) is flat and that its geometric fibres are smooth. 

We will prove the last statement first. Given a scheme $Y$, denote by  $\underline{Y}$ the associated stack. Let $\underline{t} \to \mathcal{L}og_T$ be a geometric point. This yields a morphism of fine log schemes $(t,\mathcal{M}_t) \to (T,\mathcal{M}_T)$, which need not be strict: one obtains $(t,\mathcal{M}_t)$ as the image of $\mathrm{Id} \in \underline{t}(t)$ in $\mathcal{L}og_T$. Now $(X_{t}, \mathcal{M}_{X_{t}})$ is given by the cartesian diagram



\begin{center}
\begin{tikzpicture}[auto]
\node (L1) {$(X_t,\mathcal{M}_{X_t})$};
\node (L2) [below= 1.2cm of L1] {$(t,\mathcal{M}_t)$};
\node (M1) [right= 1.2cm  of L1] {$(X,\mathcal{M}_X)$};
\node (M2) at (M1 |- L2) {$(T,\mathcal{M}_T)$};
\draw[->] (L1) to node {} (M1);
\draw[->] (L2) to node {} (M2);
\draw[->] (L1) to node  {} (L2);
\draw[->] (M1) to node {} (M2);
\end{tikzpicture} \end{center}
(in the category of fine log schemes) and the underlying scheme $X_t$ is the fibre product in the category of schemes, since the morphism $(X,\mathcal{M}_X)\rightarrow (T,\mathcal{M}_T)$ is integral. By our assumptions on the fibres, the morphism 
$
(X_t,\mathcal{M}_{X_t})\rightarrow (t,\mathcal{M}_t)
$
is  log smooth. In particular, \cite[4.6.(ii)]{olsson} implies that $\mathcal{L}og_{X_t}\rightarrow \mathcal{L}og_t$
is smooth. By \cite[3.20]{olsson}, we have an isomorphism
$$\mathcal{L}og_{X_t}\cong \mathcal{L}og_X\times_{\mathcal{L}og_T}\mathcal{L}og_t.$$ Base change along the open immersion $\underline{t}\rightarrow \mathcal{L}og_t$ shows that the induced morphism $$\mathcal{L}og_X\times_{\mathcal{L}og_T}\underline{t} \to \underline{t}$$ is indeed smooth.  Hence it remains to show that (\ref{lemm:logmap}) is flat, or that $f$ is log flat. Since log smoothness implies log flatness, this follows from \cite[Theorem 2.6.3]{gillam}.
\end{proof}

\end{document}